\documentclass[a4paper,12pt]{article}
\usepackage[top=2.5cm,bottom=2.5cm,left=2.5cm,right=2.5cm]{geometry}
\usepackage{cite, amsmath, amssymb}
\usepackage[margin=1cm,%
            font=small,%
            format=hang,%
            labelsep=period,%
            labelfont=bf]{caption}
\usepackage{amsthm,graphicx}

\newtheorem{theorem}{Theorem}[section]
\newtheorem{lemma}[theorem]{Lemma}

\newtheorem{corollary}[theorem]{Corollary}
\newtheorem{conjecture}[theorem]{Conjecture}

\newcommand{\abc}{\mathcal{ABC}}

\begin{document}

\title{The Structure of ABC-minimal Trees with Given Number of Leaves}

\author{Bojan Mohar\thanks{Supported in part by an NSERC Discovery Grant (Canada),
   by the Canada Research Chair program, and by the
    Research Grant P1--0297 of ARRS (Slovenia).}~\thanks{On leave from:
    IMFM \& FMF, Department of Mathematics, University of Ljubljana, Ljubljana,
    Slovenia.}\\[1mm]
\tt{mohar@sfu.ca}\\
  Department of Mathematics\\
  Simon Fraser University\\
  Burnaby, BC, Canada
}

\date{}

\maketitle

%%%%%%%%%%%%%%%%%%%%%%%%%%%%%%%%%%%%%%%%%%%%%%%%%%%%%%%%%%%%%%%%%%%%%%%%%%%%%%%%%%%%%%%%%%%%
%\begin{center}
%(Received June 9, 2017)
%\end{center}

\begin{abstract}
The atom-bond connectivity (ABC) index is a degree-based molecular descriptor with diverse chemical applications. Recent work of Lin et al.~\cite{Lin_etal_16} gave rise to a conjecture about the minimum possible ABC-index of trees with a fixed number $t$ of leaves. We show that this conjecture is incorrect and we prove what the correct answer is. It is shown that the extremal tree $T_t$ is unique for $t\ge 1195$, it has order $|T_t| = t + \lfloor \tfrac{t}{10}\rfloor +1$ (when $t$ mod 10 is between 0 and 4 or when it is 5, 6, or 7 and $t$ is sufficiently large) or $|T_t| = t + \lfloor \tfrac{t}{10}\rfloor + 2$ (when $t$ mod 10 is 8 or 9 or when it is 5, 6, or 7 and $t$ is sufficiently small) and its ABC-index is $\Big( \sqrt{\tfrac{10}{11}} + \tfrac{1}{10}\sqrt{\tfrac{1}{11}}\, \Big)\cdot t + O(1)$.
\end{abstract}
%%%%%%%%%%%%%%%%%%%%%%%%%%%%%%%%%%%%%%%%%%%%%%%%%%%%%%%%%%%%%%%%%%%%%%%%%%%%%%%%%%%%%%%%%%%%

\baselineskip=0.30in

\section{Introduction}

One of the most important topological indices used in Chemical Graph Theory is the Atom Bond Connectivity index, also known as the ABC-index.  It was introduced by Estrada \cite{ref7} with relation to the energy of formation of alkanes. It was extensively studied in the last few years, from the point of view of chemical graph theory \cite{ref8,ref9}, and in general graphs \cite{ref10}. Additional chemical applications of the ABC-index were discovered \cite{ref12,ref13,ref15,ref16,ref17}; we also refer to \cite{ref18,ref19,ref20} and the references therein.

The ABC-index can be defined for any graph $G$. For $v\in V(G)$, let $d_v$ denote the degree of the vertex $v$. For each edge $uv \in E(G)$, we consider its \emph{ABC-contribution}, which is the value
$$f(d_u,d_v) = \sqrt{\frac{d_u+d_v-2}{d_u d_v}}\,.$$
Then the \emph{Atom-Bond Connectivity index} (shortly \emph{ABC-index}) of $G$ is defined as the sum of ABC-contributions of its edges:
$$\abc(G) = \sum_{uv\in E(G)} f(d_u,d_v).$$
In this paper we shall only consider the ABC-indices of trees.

It is of interest to determine the extremal values of the ABC-index. In particular, the question about which trees attain the minimum possible value of the ABC-index among all trees with the given number of vertices has been extensively studied \cite{ADGH14,ref25,ref26,ref22,ref18,ref19,ref20,ref24,ref29,ref27}; see also a survey article \cite{ref28}. The solution to this long standing open question has been announced recently in \cite{AHM17}.

A similar question has been investigated in \cite{GoGu14, MaNoGu15, GoMaNoGu15, Lin_etal_15, Lin_etal_16}, with intention to figure out which trees with the given number $t$ of leaves (i.e. vertices of degree 1) have the minimum value of the ABC-index.
In this paper, a tree $T$ is said to be \emph{$t$-minimal} if $T$ has $t$ leaves and no other tree with the same number of leaves has smaller ABC-index. The problem of classifying $t$-minimal trees has been raised in \cite{GoGu14} and \cite{MaNoGu15} and was further explored in \cite{GoMaNoGu15, Lin_etal_15, Lin_etal_16}. Magnant et al.~\cite{MaNoGu15} claimed that $t$-minimal trees are balanced double stars whenever $t\ge19$. This speculation was refuted by Goubko et al.\ in \cite{GoMaNoGu15}, where $t$-minimal trees were found for all values of $t$ up to 53. Lin et al.\ \cite{Lin_etal_16} improved the knowledge about $t$-minimal trees and were able to determine $t$-minimal trees for all $t\le219$. By exploring the patterns that have shown up in their computer-aided calculations, the following conjecture was proposed.

\begin{conjecture}[Lin et al. \cite{Lin_etal_16}]
For $t\ge88$, a $t$-minimal tree has $t + \lfloor \tfrac{t}{11}\rfloor - 1$ vertices.
\end{conjecture}

Behind this conjecture there was also the precise description how the $t$-minimal trees would look like.
Unfortunately, the authors of \cite{Lin_etal_16} failed to realize that with $t$ growing, the observed patterns may still change. Here we disprove their conjecture and determine the $t$-minimal trees for every $t\ge1195$. Our main result is the following.

\begin{theorem}
\label{thm:main}
For every $t\ge 1195$, there is a unique $t$-minimal tree $T_t$ that is described in the caption of Figures \ref{fig:T_t} and \ref{fig:T_t89}. Let $r = t - 10\lfloor \tfrac{t}{10}\rfloor$. If $0\le r\le 7$ and $t$ fits the values listed at Figure \ref{fig:T_t}, then $T_t$ has $t + \lfloor \tfrac{t}{10}\rfloor + 1$ vertices and its ABC-index is equal to
$$
  \Big( \sqrt{\tfrac{10}{11}} + \tfrac{1}{10}\sqrt{\tfrac{1}{11}}\, \Big) t +
  \tfrac{9}{2}\sqrt{\tfrac{1}{11}} + \Big( 11\sqrt{\tfrac{11}{12}} + \sqrt{\tfrac{1}{12}} - \sqrt{110} + \tfrac{1}{10}\sqrt{11}\,\Big) r + O(t^{-1}).
%  - \tfrac{10}{11}\sqrt{\tfrac{1}{11}}\,\Big) r + O(t^{-1}).
$$
If\/ $5\le r\le 9$ and $t$ fits the values listed at Figure \ref{fig:T_t89}, it has $t + \lfloor \tfrac{t}{10}\rfloor + 2$ vertices and its ABC-index is equal to
$$
  \Big( \sqrt{\tfrac{10}{11}} + \tfrac{1}{10}\sqrt{\tfrac{1}{11}}\, \Big) t + \tfrac{9}{2}\sqrt{\tfrac{1}{11}} +
  9(10-r)\Big( \tfrac{14}{45}\sqrt{10} - \sqrt{\tfrac{10}{11}} - \tfrac{1}{10}\sqrt{\tfrac{1}{11}}\,\Big) + O(t^{-1}).
$$
\end{theorem}

\begin{figure}[htb]
    \centering
    \includegraphics[width=0.72\textwidth]{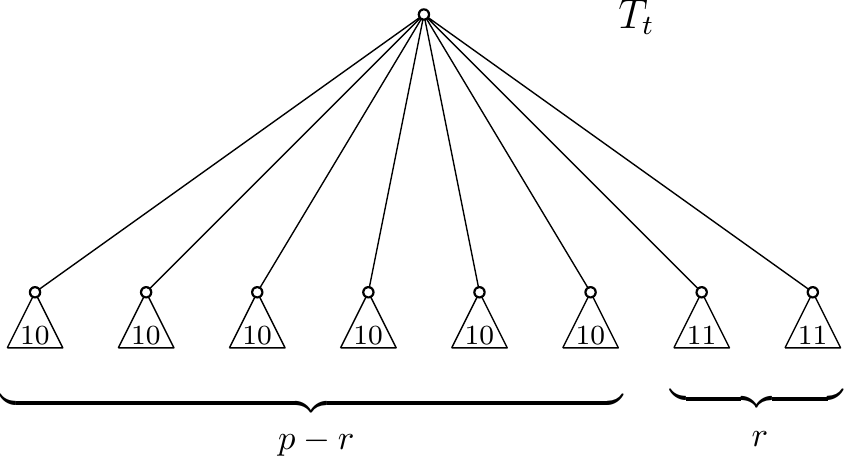}
    \caption{The unique $t$-minimal tree with $t=10p+r$ ($0\le r\le 7$) has one root vertex, $p-r$ $S_{10}$-branches and $r$ $S_{11}$-branches (see Section \ref{sect:2} for definitions) when the following holds: ($r=0$ and $t\ge1030$) or ($r\in\{1,2,3,4\}$ and $t\ge1201$) or ($r=5$ and $t\ge 1355$) or ($r=6$ and $t\ge2316$) or ($r=7$ and $t\ge7227$).}
    \label{fig:T_t}
\end{figure}

\begin{figure}[htb]
    \centering
    \includegraphics[width=0.72\textwidth]{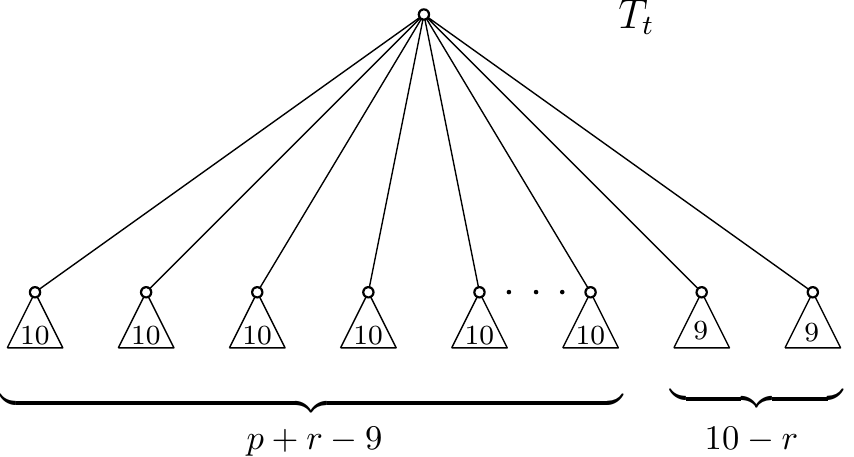}
    \caption{The unique $t$-minimal tree with $t=10p+r$ ($5\le r\le 9$) has one root vertex, $p+r-9$ $S_{10}$-branches and $10-r$ $S_{9}$-branches when ($r=5$ and $1155\le t\le 1345$) or ($r=6$ and $1106\le t\le2306$) or ($r=7$ and $1077\le t\le7217$) or ($r=8$ and $t\ge1058$) or ($r=9$ and $t\ge1039$).}
    \label{fig:T_t89}
\end{figure}

In fact, the overall bound $t\ge1195$ in Theorem \ref{thm:main} is best possible because the $t$-minimal tree for $t=1194$ is a bit different. Our computer calculations show that it is soon after $t=1000$ that we obtain structures as given in the theorem (see the caption of the figures for a detailed information), but it is only after 1195 when we obtain this for all values of $t$ modulo 10.

The ``evolution'' of the structure of $t$-minimal trees is as follows:

\begin{itemize}
\item[(i)]
For $t\le18$, $t$-minimal trees are stars.
\item[(ii)]
For $19\le t\le 35$, $t$-minimal trees are balanced double stars, see \cite{GoMaNoGu15}.
\item[(iii)]
From Goubko et al.\ \cite{GoMaNoGu15}, Lin et al.\ \cite{Lin_etal_16} and the calculations of this paper it follows that for $t\ge 36$, stars and double stars no longer occur. However, $t$-minimal trees have a mixed vertex for values of $t$ between 36 and 1194 (see definition of mixed vertices in the next section), except for those values that are indicated in captions of Figures \ref{fig:T_t} and~\ref{fig:T_t89}.
\item[(iv)]
For all values $t\ge1195$, $t$-minimal trees are described by our Theorem \ref{thm:main}.
\end{itemize}

\section{Proof of Theorem \ref{thm:main}}
\label{sect:2}

In our main proof, we will use the following results that are either well-known or easy to prove.

\begin{lemma}
\label{lem:1}
Let $k$ be a fixed positive integer. The function $f_k(d) = f(d,k)$ has the following properties:
\begin{itemize}
\item[\rm (1)]
$f(d,1)=\sqrt{1-\tfrac{1}{d}}$ is increasing in $d$ and $1-\tfrac{1}{d} < f(d,1) < 1-\tfrac{1}{2d}$.
\item[\rm (2)]
$f(d,2)=\sqrt{\tfrac{1}{2}}$ is independent of $d$.
\item[\rm (3)]
$f(d,k)=\sqrt{\tfrac{1}{k}+(1-\tfrac{2}{k})\tfrac{1}{d}\,}$ is decreasing in $d$ for every fixed $k\ge3$ and
$$\sqrt{\tfrac{1}{k}}\,\big(1+\tfrac{k-2}{2d}\big) < f(d,k) < \sqrt{\tfrac{1}{k}}\,\big(1+\tfrac{k-2}{d}\big).$$
\end{itemize}
\end{lemma}

%\begin{lemma}[\cite{ref18}]\label{lem:1a}
%Let $x,y\geq 2$ be integers and let $a\geq 0$ and $b$ $(0\leq b<y-1)$ be constants.
%Let $g(x,y)= f(x+a,y-b) - f(x,y)$. Then, $g(x,y)$ is increasing in $x$ and decreasing in $y$.
%\end{lemma}

Next we will discuss basic properties of $t$-minimal trees.

\begin{lemma}[\cite{Lin_etal_16}]
\label{lem:2}
No $t$-minimal tree has a vertex of degree 2.
\end{lemma}

\begin{lemma}[\cite{AHM17}]
\label{lem:similar}
Let $uv$ and $u'v'$ be edges of a tree $T$. Let $T_v$ ($T_{v'}$) be the component of $T-uv$ ($T-u'v'$) containing the vertex $v$ ($v'$).  Suppose that $T_v\cap T_{v'}=\emptyset$, and let $T'$ be the tree obtained from $T$ by exchanging the subtrees $T_v$ and $T_{v'}$.

{\rm (a)}
If $d_u>d_{u'}$ and $d_v<d_{v'}$, then $\abc(T)>\abc(T')$. In particular, $T$ is not minimal.

{\rm (b)}
If $d_u=d_{u'}$ or $d_v=d_{v'}$, then $\abc(T) = \abc(T')$.
\end{lemma}

Part (b) of the lemma motivates the following definition. We say that trees $T$ and $T'$ are \emph{similar} (or \emph{ABC-similar}) if $T'$ can be obtained from $T$ by a series of exchange operations, each of which satisfies one or the other equality of degrees in Lemma \ref{lem:similar}(b). Note that similarity is an equivalence relation that preserves the ABC-index.  In order to characterize $t$-minimal trees, it suffices to describe one element in each similarity class.

We will classify vertices of a tree into the following types:
\begin{itemize}
\item[\rm (L)]
A vertex of degree $1$ is a \emph{leaf}.
\item[\rm (R)]
A vertex is a \emph{root} if it is not adjacent to any leaf.
\item[\rm (S)]
A vertex of degree $d>1$ is a \emph{star vertex} if it is adjacent to $k\ge d-1$ leaves.\footnote{Note that $k=d-1$ unless the tree is a star.} A star vertex together with all adjacent leaves is a subtree and is said to be an \emph{$S_k$-branch} or an \emph{$S$-branch} of \emph{order} $k$.
\item[\rm (M)]
A vertex is a \emph{mixed vertex} if it is adjacent to at least one leaf and to at least two non-leaf vertices.
\end{itemize}

\begin{lemma}
\label{lem:4}
If $T$ is a $t$-minimal tree, then it is similar to a tree with at most one mixed vertex.
\end{lemma}

\begin{proof}
Suppose that $T$ is $t$-minimal and that it has two mixed vertices, $u$ and $u'$. If $u$ and $u'$ have different degrees, then we can make an exchange of a leaf at one of them with a larger degree subtree at the other vertex as in Lemma \ref{lem:similar}(a) and obtain a contradiction to minimality of $T$. If $d_u=d_{u'}$, then we can exchange leaves at one of these vertices with non-leaf subtrees (using the similarity exchange) and make one of $u$ and $u'$ either a star or a root vertex. By repeating this process, we arrive at a similar tree with at most one mixed vertex.
\end{proof}

In the proofs below, we will compare the ABC-index of a tree $T$ with that of a modified tree $T'$. To make the notation shorter we will write
\[\Delta(T,T') = \abc(T) - \abc(T').\]

\begin{lemma}
\label{lem:5}
If $T$ is a $t$-minimal tree, then it has at most one root.
\end{lemma}

\begin{proof}
Let us first prove that no two root vertices can be adjacent. Suppose this is not the case and that $x,y$ are two adjacent roots. Let $x_1,\dots, x_{d_x-1}$ be the neighbors of $x$ different from $y$ and let $y_1,\dots, y_{d_y-1}$ be the neighbors of $y$ different from $x$. Let $T'$ be the tree obtained from $T$ by contracting the edge $xy$ and let $w$ be the contracted vertex. Then
\begin{eqnarray*}
  \Delta(T,T') &=& f(d_x,d_y) +  \\
  & & \sum_{i=1}^{d_x-1} (f(d_x,d_{x_i})-f(d_w,d_{x_i})) + \sum_{j=1}^{d_y-1} (f(d_y,d_{y_j})-f(d_w,d_{y_j})).
\end{eqnarray*}
Each of the terms is positive by Lemma \ref{lem:1}(3). This yields a contradiction to minimality of $T$.

Thus, if there are two roots $x,y$ in $T$, there must be a mixed vertex $z$ and both roots are adjacent to $z$. We may assume that $d_x\ge d_y$. Let $z_0$ be a leaf adjacent to $z$. By using Lemma \ref{lem:similar}(a) we conclude that $d_z\le d_y$ since $z$ has a neighbor of degree 1. Also, if $d_z=d_y$, we can perform similarity exchanges at $z$ and $y$ so that $z$ becomes a root vertex and $y$ becomes a mixed vertex. However, this yields a contradiction since we would obtain a $t$-minimal tree with two adjacent roots. We conclude that $d_x\ge d_y > d_z$.

Let us now consider the degree of $x_1$. Since $d_x > d_z$, every neighbor of $z$ different from $x$ has degree at most $d_{x_1}$ by Lemma \ref{lem:similar}(a). In particular, $d_y\le d_{x_1}$.
%If equality holds, we could make a similarity exchange and obtain the situation with the root vertices $x$ and $y$ adjacent. This would be a contradiction as proved above. Thus, $d_y < d_{x_1}$.
Thus, $d_z < d_{x_1}$. However, this gives a contradiction by Lemma \ref{lem:similar}(a) since $x_1$ is adjacent to a vertex of degree 1 and $z$ is adjacent to $y$ whose degree is more than 1. This contradiction completes the proof.
\end{proof}

\begin{lemma}
\label{lem:7}
Let $T$ be a $t$-minimal tree. If a vertex $v$ has as neighbors an $S_k$-branch and an $S_l$-branch, then $|k-l|\le 1$.
\end{lemma}

\begin{proof}
Suppose that $k\ge l+2$. Let $d=d_v$. Lemma \ref{lem:similar}(a) implies that $d\ge k+1$.
By detaching a leaf from $S_k$ and attaching it to $S_l$ we obtain a tree $T'$ in which $S_k$ is replaced by $S_{k-1}$ and $S_l$ with $S_{l+1}$. We have:
\begin{eqnarray*}
  \Delta(T,T') &=&  f(d,k+1)+f(d,l+1) - f(d,k) - f(d,l+2) +  \\
   & & k f(k+1,1) + l f(l+1,1) - (k-1)f(k,1) - (l+1) f(l+2,1).
\end{eqnarray*}
For any fixed $k$ and $l<k-1$, the value of the first four terms on the right-hand side of the above equality is increasing in terms of $d$, so it suffices to treat the case when $d$ is minimum possible, $d=k+1$. Then, for any fixed $k$, the value is decreasing in $l$, so we may assume that $l=k-2$. In other words,
\begin{eqnarray*}
  \Delta(T,T') &\ge&  f(k+1,k+1) + f(k+1,k-1) - 2 f(k+1,k)+ \\
  && k f(k+1,1) + (k-2) f(k-1,1) - 2(k-1)f(k,1).
\end{eqnarray*}
The right-hand side is decreasing in terms of $k$ and is 0 in the limit when $k\to \infty$. This shows that $\Delta(T,T') > 0$, which is a contradiction to the $t$-minimality of $T$.
\end{proof}

When $t\le18$, $t$-optimal trees are stars. For $19\le t\le 35$ then they become balanced double stars, see \cite{GoMaNoGu15}. It follows from Goubko et al. \cite{GoMaNoGu15} and Lin et al. \cite{Lin_etal_16} that for $t\ge 36$, stars and double stars no longer occur. Therefore, there exists a root or a mixed vertex. As a corollary, our lemmas above imply the following.

\begin{corollary}
\label{cor:structure intermediate}
If $T$ is a $t$-minimal tree and $t\ge 36$, then $T$ has either one mixed vertex, one root, or one mixed vertex and one root that are adjacent, and all other vertices are stars and leaves.
\end{corollary}

\begin{figure}[htb]
    \centering
    \includegraphics[width=1.0\textwidth]{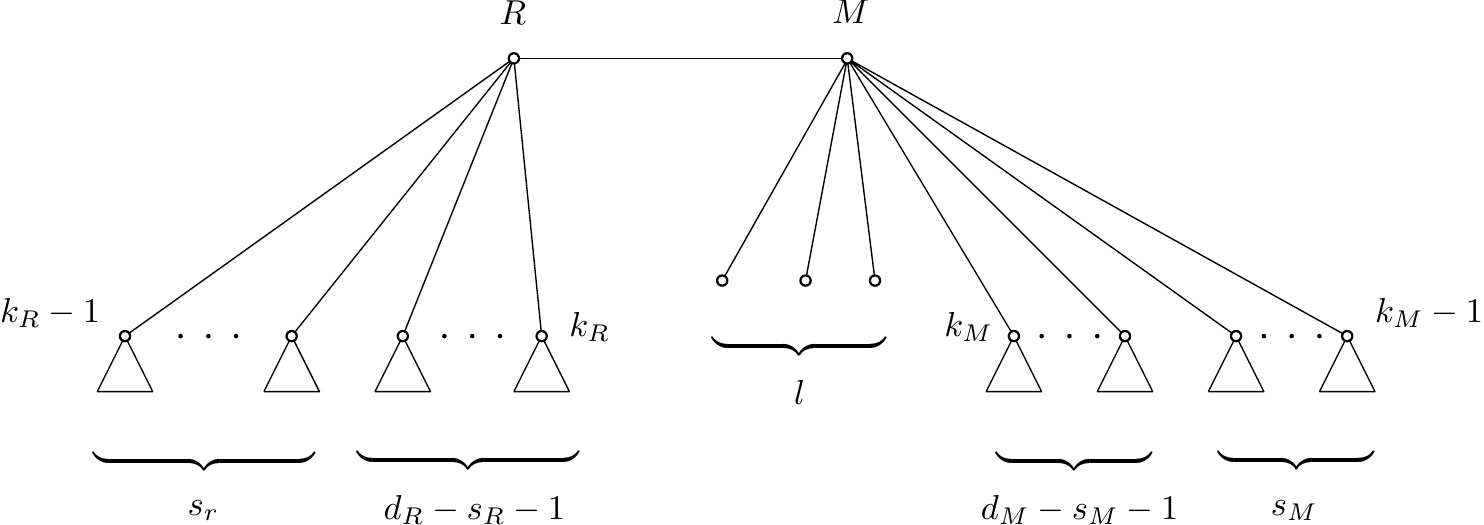}
    \caption{The $t$-minimal trees have at most one root vertex ($R$) and at most one mixed vertex ($M$).}
    \label{fig:RM}
\end{figure}

The results proved above give a restricted structure for $t$-minimal trees. The structure with both---a root and a mixed vertex---is shown in Figure \ref{fig:RM}.
There may be just one root vertex (which turns out to be the case when $t\ge 1195$) or just one mixed vertex (which happens when $t<1195$ with several exceptions). The $S$-branches adjacent to $R$ or to $M$ have the same order or two consecutive orders (see Lemma \ref{lem:7}). So we have up to 7 parameters: degrees $d_R$ and $d_M$ of $R$ and $M$, the number of leaves $l$ at the mixed vertex, the larger order $k_R$ and $k_M$ of $S$-branches at $R$ and at $M$, respectively, and the number $s_R$ and $s_M$ of $S$-branches of order $k_R-1$ and $k_M-1$. Clearly,
$$
   t = l + (d_R-1)k_R - s_R + (d_M-l-1)k_M - s_M.
$$
In addition to this, we may assume (after using similarity exchanges if needed) to have the following inequalities (assuming $R$ and $M$ both exist):
\begin{eqnarray*}
  d_R &\ge& d_M \ge k_R+1\ge k_M+1 \\
  k_R &\ge& k_M \\
  l &\le& d_M-2 \\
  s_R &\le& d_R-2 \\
  s_M &\le& d_M-l-2.
\end{eqnarray*}
Moreover, as proved in \cite{Lin_etal_16}, $k_M\ge5$. This enables us to make a brute force search for optimal parameters for every fixed $t$.\footnote{When $t\le 2000$, this takes only a couple of seconds on a desktop PC.}  Additional restrictions provided below reduce the number of cases to be treated and also establish our main theorem for large enough $t$. We will use the above notation in the rest of the paper. We shall also assume that $t\ge36$.

Before continuing, we define the notion of the \emph{ABC-contribution} $c(v)$ of a leaf $v$ in a tree $T$ as follows.
If $v$ is contained in an $S_k$-branch and the star vertex is adjacent to a root or to a mixed vertex of degree $d$, then
$c(v) = f(k+1,1) + \frac{1}{k} f(k+1,d)$. If $v$ is a leaf adjacent to $M$ and $R$ exists, then $c(v) = f(d_M,1) + \frac{1}{l} f(d_M,d_R)$.  The remaining possibility is that $v$ is a leaf adjacent to $M$ and $R$ does not exists; in that case $c(v) = f(d_M,1)$. Clearly, the sum of all contributions of the leaves satisfies:
\begin{equation}
   \sum_{v; \deg(v)=1} c(v) = \abc(T).
\label{eq:leaf contributions 1}
\end{equation}

The leaf contributions of leaves adjacent to a star of order $k$ only depend on $k$ and on the degree $d$ of the root or mixed vertex adjacent to the star. In that case we also use the notation
$$c(k,d) = f(k+1,1) + \frac{1}{k} f(k+1,d).$$
We will need the following basic properties of leaf contributions.

\begin{lemma}
\label{lem:5a}
The leaf contributions function $c(k,d)$ has the following properties.

{\rm (a)} $c(k,d)$ is monotone decreasing in $d$ for every fixed $k\ge2$.

{\rm (b)} $c(k,d) - c(k_0,d)$ is monotone increasing in $d$ when $k>k_0$ and decreasing when $k<k_0$.

{\rm (c)} The function $\Delta_0(k,d) = k\cdot c(k,d) - (k+1) c(k+1,d)$ is increasing in $d$.
\end{lemma}

\begin{proof}
(a) follows easily from Lemma \ref{lem:1}(3). To prove (b) and (c), one simply looks at the derivatives
$$
    \frac{\partial}{\partial d} c(k,d) = \frac{1}{k}\, \frac{\partial}{\partial d} f(k+1,d) = \frac{-(1-\tfrac{2}{k+1})}{2kd^2f(k+1,d)}
$$
and
$$
    \frac{\partial}{\partial d} (k\,c(k,d)) = \frac{-(1-\tfrac{2}{k+1})}{2d^2f(k+1,d)}.
$$
For (b) it suffices to consider the case when $k_0=k-1$ since $c(k,d)-c(k_0,d)$ can be written as the sum of consecutive differences of the form $c(i,d)-c(i-1,d)$ for $i=k_0+1,\dots , k$. As this is a simple exercise, we leave the details to the reader.
\end{proof}

Table \ref{tab:1} shows contributions $c(k,d)$ of leaves in $S$-branches of order $k$ for $k=5,\dots,16$ when the degree $d$ of the adjacent root or mixed vertex is 120 or very large (respectively). The differences between the values $c(k,d)-c(10,d)$ change when $d$ gets larger, but they stay between the two values in the table by Lemma \ref{lem:5a}(b). The minimum of $c(k,d)$ when $d\ge120$ is fixed is always attained at $k=10$.

\begin{table}[ht]
\centering % used for centering table
\begin{tabular}{|c||c|c||c|c|}
  \hline
$k$ & $c(k,120)$ & $c(k,120)-c(10,120)$ & $c(k,\infty)$ & $c(k,\infty)-c(10,\infty)$ \\ \hline
5 & 0.99587026 & 0.011146309 & 0.99452072 & 0.010906887 \\
6 & 0.99011316 & 0.005389211 & 0.98881431 & 0.005200473 \\
7 & 0.98716926 & 0.002445313 & 0.98592210 & 0.002308263 \\
8 & 0.98567376 & 0.000949811 & 0.98447583 & 0.000861993 \\
9 & 0.98497203 & 0.000248084 & 0.98381983 & 0.000205997 \\
10 & 0.98472395 & 0 & 0.9836138 & 0 \\
11 & 0.98474189 & 0.000017939 & 0.9836704 & 0.000056574 \\
12 & 0.98491753 & 0.000193580 & 0.9838815 & 0.000267700 \\
13 & 0.98518611 & 0.000462157 & 0.9841828 & 0.000568935 \\
14 & 0.98550786 & 0.000783911 & 0.9845347 & 0.000920824 \\
15 & 0.98585791 & 0.001133961 & 0.9849126 & 0.001298763 \\
16 & 0.98622049 & 0.001496542 & 0.9853011 & 0.001687235 \\
  \hline
\end{tabular}
\caption{Leaf contributions at $S_k$-branches (with the last shown digit rounded).}
\label{tab:1}
\end{table}

\begin{lemma}
\label{lem:6}
Let $T$ be a $t$-minimal tree, where $t\ge 200$.

{\rm (a)} If $T$ contains a root vertex, then the root is adjacent with an $S_k$-branch, where $k\le10$. Consequently, any $S_q$-branch in $T$ has $q\le 11$.

{\rm (b)} If $T$ contains a mixed vertex, then it does not have a root vertex.
\end{lemma}

\begin{proof}
(a) If $T$ contains a root and a mixed vertex, then by using Lemma \ref{lem:similar}(a) we see that $T$ is similar to a minimal tree with $d_R \ge d_M \ge k_R+1\ge k_M+1$, which we assume henceforth. This implies that more than $t/2$ leaves are within $S$-branches that are adjacent to the root $R$. To verify the claim, it suffices to see that there exists an $S$-branch of order $10$ or less. Suppose, for a contradiction, that this is not the case. Then the $S$-branches adjacent to $R$ have orders $k_1\ge k_2\ge \cdots \ge k_{d_R-1}\ge11$. Since $\sum_i k_i > \frac{t}{2} \ge 100$, it is easy to see that there is an index $j$ such that $\sum_{i=1}^j k_i \ge 10(j+1)$. Then we may change $T$ to a tree with the same number of leaves by replacing the $S_{k_i}$-branches ($i=1,\dots,j$) with $j+1$ branches, one of order 10 and the others of orders $k_i'$, where $10\le k_i'\le k_i$ for each $i=1,\dots,j$. It is an easy calculation to show that the leaf contributions of all the leaves in $S$-branches adjacent to $R$ decrease. The degree of $R$ also increases, so the contribution of the possible edge $RM$ also drops (by Lemma \ref{lem:1}(3)), and all other contributions remain unchanged. Therefore, the ABC-index drops, which is a contradiction.

(b) Suppose that $T$ has a root $R$ and a mixed vertex $M$. We will reach a contradiction in two steps. In the first step we show that $l < 37$.

Suppose that $l\ge37$. We change the tree $T$ into a tree $T'$ which has one root vertex by moving all stars adjacent to $M$ to be adjacent to the root vertex and replacing $l$ pendant edges at $M$ with stars of orders 9, 10, 11 and make them adjacent to $R$ as well. By applying formula (\ref{eq:leaf contributions 1}) and Lemma \ref{lem:1}, we see that
$$
    \Delta(T,T') > l f(d_M,1) - l \max_{k\in\{9,10,11\}} c(k,d_R+4) \ge l(f(l+2,1) - c(9,l+4)).
$$
It is easy to see that the factor $f(l+2) - c(9,l+4)$ is positive for every $l\ge 37$. This contradicts $t$-minimality of $T$ and proves that $l\le36$.

Knowing that $l\le36$, we continue as follows. First, recall that $d_R \ge d_M \ge l+2$. Thus, there are $l$ stars $S_1,\dots, S_l$ adjacent to the root vertex. Each of them is of order $\le11$ by part (a). Now we change $T$ into a tree $T'$ as follows.
We first move all stars adjacent to $M$ to the root vertex (which will be denoted by $R'$) and then, for $i=1,\dots,l$, replace each star $S_i$ of order $k_i$ with a star $S_i'$ of order $k_i+1$. Then we remove the vertex $M$ and adjacent leaves. The resulting tree $T'$ has the same number of leaves and we will show that its ABC-index is smaller (thus giving a contradiction). The edge-contributions of leaves in all stars different from $S_1,\dots, S_l$ have gone down (or stayed the same) because $d_{R'} \ge d_R$ (see Lemma \ref{lem:5a}(a)). Thus
$$
   \Delta(T,T') \ge \sum_{i=1}^l \Bigl( c(w_i) + k_i c(k_i,d_R) - (k_i+1) c(k_i+1, d_{R'}) \Bigr)
$$
where $w_1,\dots,w_l$ are the leaves adjacent to $M$ in $T$. We claim that each term in the above sum is positive.

First of all, we have:
\begin{eqnarray*}
  c(w_i) &=& f(d_M,1)+ \tfrac{1}{l} f(d_M,d_R) \\
         &\ge& f(d_M,1) + \tfrac{1}{d_M-2} f(d_M,\infty) \\
         &=& f(d_M,1) + \tfrac{1}{d_M-2} \sqrt{1/d_M}.
\end{eqnarray*}
By using Lemma \ref{lem:5a}(a) and (c) we conclude the following:
\begin{eqnarray*}
   k_i c(k_i,d_R) - (k_i+1) c(k_i+1, d_{R'}) &\ge& k_i c(k_i,d_{R'}) - (k_i+1) c(k_i+1, d_{R'}) \\
   &=& \Delta_0(k_i,d_{R'})\ge \Delta_0(k_i,d_R).
\end{eqnarray*}
By combining the above three inequalities, we see that it suffices to prove the following inequality for $d=d_M$ and $k=k_i\le11$:
\begin{equation}\label{eq:noRandM}
  f(d,1) + \sqrt{\tfrac{1}{d(d-2)^2}} + \Delta_0(k,d) > 0.
\end{equation}

Inequality (\ref{eq:noRandM}) has been verified for intermediate values of $d$, $26\le d\le 100$, and for all $k$ from 5 to 11 by a computer calculation. For $d\ge100$ it can be proved analytically (the value is increasing in $d$ for $d\ge 100$ for each $k$). The details are left to the reader.

Unfortunately, (\ref{eq:noRandM}) fails for some values of $k\le11$ when $d\le 25$.
This range can be considered by using different approaches. We found it easiest to use our afore-mentioned algorithm to verify the claim by computer for all $t\le 1500$. We may then assume that $t\ge 1500$. It is easy to see that in this case we have $d_{R'}\ge 120$. Moreover, by the same proof as used in part (a) of the proof, we see that there are more than $l$ $S_{10}$-stars. Therefore, we only need to treat the case where $k=10$. The inequality (\ref{eq:noRandM}) can be replaced by:
\begin{equation}\label{eq:noRandM2}
  f(d,1) + \sqrt{\tfrac{1}{d(d-2)^2}} + \Delta_0(10,120) > 0.
\end{equation}
It is easy to see that (\ref{eq:noRandM2}) holds for every $d$, $3\le d\le 25$. This completes the proof.
\end{proof}

\begin{lemma}
\label{lem:6a}
If $T$ is a $t$-minimal tree, where $t\ge1195$, then $T$ has a root vertex and does not have a mixed vertex, and $T$ is isomorphic to the tree $T_t$ as defined in Figures \ref{fig:T_t} and~\ref{fig:T_t89}.
\end{lemma}

\begin{proof}
First, we claim that $T$ does not have a mixed vertex. If it does, there is no root vertex by the previous lemma.
Let $M$ be the mixed vertex. Previous results imply that $d_M > 120$.\footnote{This conclusion needs some case analysis when $t$ is very close to 1195. Independently, we have checked our claims for small values of $t$ by computer and thus this is not really an issue to be worried about.}
Let $d'$ be the number of $S_{10}$-stars adjacent to $M$ and let $r$ be the number of $S_9$-stars or $S_{11}$-stars. Let $d=d_M-1=d'+r+l-1\ge d'+r$. The case when there are stars of order 9 is much easier to argue (by the same proof method as used below), thus we shall assume for brevity that we have no stars of order 9. It is easy to see that we cannot have ten or more $S_{11}$-stars (replacing 10 of them with 11 stars of order 10 decreases the ABC-index when $d_M>120$). Thus $0\le r\le 9$.

The following inequality which holds for every $d\ge120$ is easy to verify:
\begin{equation}\label{eq:10vs11}
  f(11,d+1)-f(11,d) + f(12,d+1)-f(12,d) \le 6\cdot 10^{-4}.
\end{equation}

Now we consider the following tree $T'$ with $t$ leaves. We remove one of the leaves adjacent to $M$ and change one $S_{10}$ into $S_{11}$. The following chain of inequalities use the following: Lemma \ref{lem:1}(a) for the first inequality, (\ref{eq:10vs11}) and $d'+r \le d$ and $f(11,d+1)-f(11,d) < 0$ and $r\le 9$ for the second inequality:
\begin{eqnarray*}
  \Delta(T,T') &=& 10f(1,11)+f(1,d+1)-11f(1,12) + \\
   &&  (l-1)(f(1,d+1)-f(1,d)) + \\
   &&  d'f(11,d+1)-(d'-1)f(11,d) + \\
   &&  r f(12,d+1) - (r+1)f(12,d) \\
   &\ge& 10f(1,11)+f(1,d+1)-11f(1,12) + f(11,d+1) + \\
   &&  (d'-1)(f(11,d+1)-f(11,d)) + \\
   &&  r (f(12,d+1) - f(12,d)) - f(12,d) \\
   &\ge& 10f(1,11)+f(1,d+1)-11f(1,12) + f(11,d+1) + \\
   &&  (d-1)(f(11,d+1)-f(11,d)) - f(12,d) + 0.0054.
\end{eqnarray*}
The last quantity above is positive for $d=120$ and only increases\footnote{This is easily tabulated for small values of $d\ge120$. For the general case, basic calculus can be used to prove that the function is increasing.} when $d$ grows.
This implies that $\Delta(T,T')>0$. This contradiction shows that $T$ must have a root vertex and not a mixed vertex.

To deal with the case when there is a root vertex, we look at the contributions of the leaves. The minimum contribution is achieved with $S$-branches of order 10, the next smallest values are 11 and 9 (see Table \ref{tab:1}), and by Lemma \ref{lem:7} we have only one of these. This means that the tree is one of those described in captions of Figures \ref{fig:T_t} and \ref{fig:T_t89}. This has been done by a computation for all values of $t\le 10000$; the threshold values stated in Figures \ref{fig:T_t} and \ref{fig:T_t89} have been obtained from these computations. For $t\ge 10000$ it suffices to prove that the following inequality (which compares the ABC-indices of trees from both figures) is satisfied for $0\le r\le 7$ and for $p=\lfloor t/10\rfloor \ge 1000$:
$$
   10(p-r)c(10,p) + 11r\, c(11,p) < 10(p+r-9)c(10,p+1) + 9(10-r) c(9,p+1)
$$
and that the reverse inequality holds when $r=8,9$. Again, this task reduces to verify the inequality for $p=1000$ and then show that the difference is decreasing and that for $r=8,9$ the difference stays positive (by considering the limit when $p\to\infty$). It is easy to see that the worst cases are when $r=7$ and $r=8$ and that the proof for these two values implies the proof for all other values of $r$.
\end{proof}

\begin{lemma}
\label{lem:10}
Suppose that $t=10p+r$, where $0\le r\le 9$.

{\rm (a)} Let $T_t$ be the tree as defined by Figure \ref{fig:T_t}. Then
\begin{eqnarray*}
   \abc(T_t) &=&
  \left( \sqrt{\tfrac{10}{11}} + \tfrac{1}{10}\sqrt{\tfrac{1}{11}} \right) t +
  \tfrac{9}{2}\sqrt{\tfrac{1}{11}} + \\
    && \left(11\sqrt{\tfrac{11}{12}} + \sqrt{\tfrac{1}{12}} - \sqrt{110} + \tfrac{1}{10}\sqrt{11}\,\right) r + O(t^{-1}).
\end{eqnarray*}

{\rm (b)} Let $T_t$ be the tree as defined by Figure \ref{fig:T_t89}. Then
\begin{eqnarray*}
   \abc(T_t) &=&
      \Big( \sqrt{\tfrac{10}{11}} + \tfrac{1}{10}\sqrt{\tfrac{1}{11}}\, \Big) t + \tfrac{9}{2}\sqrt{\tfrac{1}{11}} + \\
   && 9(10-r)\Big( \tfrac{14}{45}\sqrt{10} - \sqrt{\tfrac{10}{11}} - \tfrac{1}{10}\sqrt{\tfrac{1}{11}}\,\Big) + O(t^{-1}).
\end{eqnarray*}
\end{lemma}

\begin{proof}
The formula (\ref{eq:leaf contributions 1}) implies, for the first case, that
$$
   \abc(T_t) = 10(p-r)c(10,p) + 11r\, c(11,p)
$$
and for the second one:
$$
   \abc(T_t) = 10(p+r-9)c(10,p+1) + 9(10-r) c(9,p+1).
$$
A routine calculation using approximations
$$
    f(k,d) = \sqrt{\tfrac{1}{k}}\Bigl(1 + \tfrac{k-2}{2d} - O(d^{-2})\Bigr)\quad \hbox{and} \quad d=\tfrac{t}{10} + O(1)
$$
(see Lemma \ref{lem:1}(3)) gives the claimed expressions.
\end{proof}

\begin{proof}[Proof of Theorem \ref{thm:main}]
The structure of $t$-minimal trees described after Corollary \ref{cor:structure intermediate} enabled us to search for minimal trees for all values of $t\le 2000$. The calculations verify the claim of the theorem. On the other hand, for $t\ge 2000$, the results given above show the same: we have one root vertex and any $t$-minimal tree is isomorphic to $T_t$.
Finally, Lemma \ref{lem:10} gives the asymptotic value of $\abc(T_t)$.
\end{proof}

\subsection*{Acknowledgement}
The author is indebted to Seyyed Aliasghar Hosseini for several helpful remarks.


\begin{thebibliography}{99}

\bibitem{ADGH14}
M. B. Ahmadi, D. Dimitrov, I. Gutman, S. A. Hosseini,
Disproving a conjecture on trees with minimal atom-bond connectivity index,
{\it MATCH Commun. Math. Comput. Chem.\/} {\bf 72} (2014) 685--698.

\bibitem{ref25} M. B. Ahmadi, S. A. Hosseini, P. Salehi Nowbandegani,
On trees with minimal atom bond connectivity index,
{\it MATCH Commun. Math. Comput. Chem.\/} {\bf 69} (2013) 559--563.

\bibitem{ref26} M. B. Ahmadi, S. A. Hosseini, M. Zarrinderakht,
On large trees with minimal atom-bond connectivity index,
{\it MATCH Commun. Math. Comput. Chem.\/} {\bf 69} (2013) 565--569.

\bibitem{ref22} J. Chen, X. Guo,
Extreme atom-bond connectivity index of graphs,
{\it MATCH Commun. Math. Comput. Chem.\/} {\bf 65} (2011) 713--722.

\bibitem{ref12} J. Chen, J. Liu, Q. Li,
The atom-bond connectivity index of catacondensed polyomino graphs,
{\it Discrete Dyn. Nat. Soc.\/} {\bf 2013} (2013) ID 598517.

\bibitem{ref10} Kinkar Ch. Das,
Atom-bond connectivity index of graphs.
{\it Discrete Appl.\ Math.\/} {\bf 158} (2010) 1181--1188.

\bibitem{ref13} K. C. Das, N. Trinajsti\'{c},
Comparison between first geometric-arithmetic index and atom-bond connectivity index,
{\it Chem. Phys. Lett.\/} {\bf 497} (2010) 149--151.

\bibitem{ref18} D. Dimitrov,
On structural properties of trees with minimal atom-bond connectivity index,
{\it Discrete Appl.\ Math.\/} {\bf 172} (2014) 28--44.

\bibitem{ref19} D. Dimitrov,
On structural properties of trees with minimal atom-bond connectivity index II: Bounds on B1- and B2-branches,
{\it Discrete Appl.\ Math.\/} {\bf 204} (2016) 90--116.

\bibitem{ref20} D. Dimitrov, Z. Du, C.~M. da Fonseca,
On structural properties of trees with minimal atom-bond connectivity index III: Trees with pendent paths of length three,
{\it Applied Mathematics and Computation\/} {\bf 282} (2016) 276--290.

\bibitem{ref7} Ernesto Estrada,
Atom-bond connectivity and the energetic of branched alkanes.
{\it Chem.\ Phys.\ Lett.\/} {\bf 463} (2008) 422--425.

\bibitem{ref8}
Boris Furtula, Ante Graovac, and Damir Vuki\v{c}evi\'{c},
Atom-bond connectivity index of trees.
{\it Discrete Appl.\ Math.\/} {\bf 157} (2009) 2828--2835.

\bibitem{GoGu14}
M. Goubko, I. Gutman,
Degree-based topological indices: Optimal trees with given number of pendents,
{\it Appl. Math. Comput.\/} {\bf 240} (2014) 387--398.

\bibitem{GoMaNoGu15}
M. Goubko, C. Magnant, P. Salehi Nowbandegani, I. Gutman,
ABC index of trees with fixed number of leaves,
{\it MATCH Commun. Math. Comput. Chem.\/} {\bf 74} (2015) 697--702.

\bibitem{ref24} I. Gutman, B. Furtula,
Trees with smallest atom-bond connectivity index,
{\it MATCH Commun. Math. Comput. Chem.\/} {\bf 68} (2012) 131--136.

\bibitem{ref28}
I. Gutman, B. Furtula, M. B. Ahmadi, S. A. Hosseini, P. Salehi Nowbandegani, M. Zarrinderakht,
The ABC index conundrum,
{\it Filomat\/} {\bf 27} (2013) 1075--1083.

\bibitem{ref29} I. Gutman, B. Furtula, M. Ivanovi\'c,
Notes on trees with minimal atom-bond connectivity index,
{\it MATCH Commun. Math. Comput. Chem.\/} {\bf 67} (2012) 467--482.

\bibitem{ref15}
I. Gutman, J. To\v{s}ovi\'{c}, S. Radenkovi\'{c}, S. Markovi\'{c},
On atom-bond connectivity index and its chemical applicability,
{\it Indian J. Chem.\/} {\bf 51A} (2012) 690--694.

\bibitem{AHM17}
S. A. Hosseini, B. Mohar, and M. B. Ahmadi,
The evolution of the structure of ABC-minimal trees,
submitted for publication, 2017.

\bibitem{ref16} X. Ke,
Atom-bond connectivity index of benzenoid systems and uoranthene congeners,
{\it Polycycl. Aromat. Comp.\/} {\bf 32} (2012) 27--35.

\bibitem{Lin_etal_16}
W. Lin, J. Chen, C. Ma, Y. Zhang, J. Chen, D. Zhang, and F. Jia,
On trees with minimal ABC index among trees with given number of leaves,
{\it MATCH Commun. Math. Comput. Chem.\/} {\bf 76} (2016) 131--140.

\bibitem{ref27} W. Lin, X. Lin, T. Gao, X. Wu,
Proving a conjecture of Gutman concerning trees with minimal ABC index,
{\it MATCH Commun. Math. Comput. Chem.\/} {\bf 69} (2013) 549--557.

\bibitem{Lin_etal_15}
W. Lin, C. Ma, Q. Chen, J. Chen, T. Gao, B. Cai,
Parallel search trees with minimal ABC index with MPI+OpenMP,
{\it MATCH Commun. Math. Comput. Chem.\/} {\bf 73} (2015) 337--343.

\bibitem{MaNoGu15}
C. Magnant, P. Salehi Nowbandegani, I. Gutman,
Which tree has the smallest ABC index among trees with $k$ leaves?
{\it Discr. Appl. Math.\/} {\bf 194} (2015) 143--146.

\bibitem{ref9} Rundan Xing, Bo Zhou, and Zhibin Du,
Further results on atom-bond connectivity index of trees.
{\it Discrete Appl.\ Math.\/} {\bf 158} (2010) 1536--1545.

\bibitem{ref17} J. Yang, F. Xia, H. Cheng,
The atom-bond connectivity index of benzenoid systems and phenylenes,
{\it Int. Math. Forum\/} {\bf 6} (2011) 2001--2005.

\end{thebibliography}
\end{document}